\documentclass[12pt,twoside]{article}
\usepackage[all]{}
\usepackage {amssymb,latexsym,amsthm,amscd,amsmath}
\newtheorem{theorem}{Theorem}[section]

\newtheorem{corollary}{Corollary}[section]
\newtheorem{proposition}{Proposition}[section]

\newtheorem*{theorem*}{Theorem}
\theoremstyle{definition}

\numberwithin{equation}{subsection}

\newcommand{\ignore}[1]{}

\newcommand{\mynote}[1]{}
\begin{document}
\setcounter{section}{0}
\title{\bf On $R$-triviality of $F_4$}
\author{Maneesh Thakur }
\date{}
\maketitle
\begin{abstract}
\noindent
\it{It is known that simple algebraic groups of type $F_4$ defined over a field $k$ are precisely the full automorphism groups of Albert algebras over $k$. We explore $R$-triviality for the group $\text{\bf Aut}(A)$ when $A$ is an Albert algebra. In this paper, we consider the case when $A$ is an Albert division algebra, that arises from the first Tits construction. We prove that $\text{\bf Aut}(A)$ is $R$-trivial, in the sense of Manin. }
\end{abstract} 

\section{\bf Introduction}
The aim of this paper is to start an investigation on the rationality properties of simple algebraic groups of type $F_4$, defined over fields.
One knows that for a field $k$, simple algebraic groups of type $F_4$ are precisely the full groups of automorphisms of Albert algebras (exceptional Jordan algebras) over $k$. Recently there has been a flurry of activity in studying rationality properties of algebraic groups. For example, in the papers (\cite{Gar, G, P-T-W, Pra, Th-1, Th-2, Th-3}, \cite{A-C-P-2}), the Kneser-Tits problem was settled in the affirmative, via $R$-triviality of certain simple algebraic groups, or via the computation of the group of $R$-equivalent ponts for the algebraic group. The fundamental work of (\cite{G}) connects Whitehead groups of an isotropic algebraic group $G$ with the quotient $G(k)/R(G(k))$, where $R(G(k))$ is the normal subgroup of $G(k)$ of points $R$-equivalent to the identity element of $G$. Hence, it reduces the Kneser-Tits problem, or the problem of computing the Whitehead group, to computing the quotient $G(k)/R(G(k))$. In this paper, we take up the case of first Tits construction Albert algebras. If the algebra is reduced and arises from first construction, one knows that $G=\text{\bf Aut}(A)$ is split over $k$, hence is $R$-trivial. We take up the case of Albert division algebras that arise from the first Tits construction in this paper. The case of reduced Albert algebras which contain non-zero nilpotents corresponds to isotropic group of type $F_4$ and they are known to be $R$-trivial. The case of groups of type $F_4$ that correspond to reduced Albert algebras without nilpotents and those that correspond to Albert division algebras that are pure second Tits constructions, is work in progress.

We would like to mention that at the time this preprint was completed, S. Alsaody, V. Chernousov and A. Pianzola posted a preprint on the arXiv (\cite{A-C-P-1}), proving the R-triviality for groups of type $F_4$ that arise from the first Tits constructions, over fields of characteristic different from $2$ and $3$. Their proof relies on certain cohomological techniques. Our paper proves the $R$-triviality of algebraic groups of type $F_4$ defined over fields of arbitrary characteristic, while the methods are via studying $R$-triviality for certain subgroups of the automorphism groups as well as the structure groups of Albert algebras. 

\section{\bf Preliminaries} We refer the reader to (\cite{P2}), (\cite{SV}) and (\cite{KMRT}) for basic material on Albert algebras. In this section, we quickly introduce some notions related to Albert algebras that are indispensable. All base fields will be 
assumed to be infinite of arbitrary characteristic unless specified otherwise. The matrial below is being reproduced from (\cite{Th-2}) for convenience. \\
\noindent
For this, we refer to (\cite{PR7} or \cite{PR5}). Let $J$ be a finite dimensional vector space over a field $k$. A \emph{cubic norm structure} on $J$ consists of a triple $(N,\#, c)$, where $c\in J$ is a base point, called the \emph{identity element} of the norm structure and 
\begin{enumerate} 
\item $N:J\rightarrow k$ is a \emph{cubic form} on $J$,
\item $N(c)=1$,
\item the \emph{trace form} $T:J\times J\rightarrow k$, defined by $T(x,y):=-\Delta^x_c\Delta^y log N$, is nondegenerate, 
\item the \emph{adjoint} map $\#:J\rightarrow J$ defined by $T(x^{\#},y)=\Delta^y_x N$, is a quadratic map such that 
\item $x^{\#\#}=N(x)x$,
\item $c^{\#}=c$,
\item $c\times x=T(x)c-x$, where $T(x):=T(x,c)$ is the \emph{trace} of $x$ and $x\times y:=(x+y)^{\#}-x^{\#}-y^{\#}$,
\end{enumerate}
and these conditions hold in all scalar base changes of $J$. Here $\Delta^y_xf$ is the \emph{directional derivative} of a polynomial function $f$ on $J$, in the direction $y$, evaluated at $x$ and $\Delta^y log f$ denotes the \emph{logarithmic derivative} of $f$ in the direction $y$. For details, we refer to (\cite{J1}, Chap. VI, Sect. 2).
Let $x\in J$. Define  
$$U_x(y):=T(x,y)x-x^{\#}\times y,~y\in J.$$ 
Then with $1_J:=c$ and $U_x$ as above, $J$ is a unital \emph{quadratic Jordan algebra} (see \cite{McK}), denoted by $J(N,c)$. The (linear) operators $U_x$ are called the \emph{$U$-operators} of $J$. An element $x\in J$ is defined to be \emph{invertible} if $N(x)\neq 0$ and $x^{-1}:=N(x)^{-1}x^{\#}$. The structure $J(N,c)$ is a \emph{division algebra} if $U_x$ is surjective for all $x\neq 0$, or equivalently, $N(x)\neq 0$ for all $x\neq 0$. Special Jordan algebras of degree $3$ provide imporant class of examples, we list them below for our purpose: more precisely.\\
\vskip1mm
\noindent
{\bf Example.} Let $D$  be a separable associative algebra over $k$ of degree $3$. Let $N_D$ denote its norm and $T_D$ the trace. Let $\#:D\rightarrow D$ be the adjoint map. Then $(N_D, \#, 1_D)$ is a cubic norm structure, where $1_D$ is the unit element of $D$. We get a quadratic Jordan algebra structure on $D$, which we will denote by $D_+$. \\
\noindent
 Let $(B,\sigma)$ be a separable associative algebra over $k$ with an involution $\sigma$ of the second kind (over its center). With the unit element $1$ of $B$ and the restriction of the norm $N_B$ of $B$ to $(B,\sigma)_+:=\{b\in B|\sigma(b)=b\}$, we obtain a cubic norm structure and hence a Jordan algebra structure on $(B,\sigma)_+$ which is a substructure of $B_+$. \\
\noindent
\vskip1mm
\noindent
    {\bf Tits process :} Let $D$ be a finite dimensional associative $k$-algebra of degree $3$ with norm $N_D$ and trace $T_D$. Let $\lambda\in k^{\times}$. On the $k$-vector space
    $D\oplus D\oplus D$, we define a cubic norm structure as below.  
$$1:=(1,0,0),~N((x,y,z)):=N_D(x)+\lambda N_D(y)+\lambda^{-1}N_D(z)-T_D(xyz),$$
$$(x,y,z)^{\#}:=(x^{\#}-yz,\lambda^{-1}z^{\#}-xy,\lambda y^{\#}-zx).$$  
    The Jordan algebra associated to this norm structure is denoted by $J(D,\lambda)$. The algebra $D_+$ is a subalgebra of $J(D,\lambda)$ through the first summand.
    Then $J(D,\lambda)$ is a division algebra if and only if $\lambda\notin N_D(D)$ (see \cite{PR7}, 5.2). This construction is called the \emph{first Tits process} arising from the parameters $D$ and $\lambda$. 

    Let $K$ be a quadratic \'{e}tale extension of $k$ and $B$ a separable associative algebra of degree $3$ over $K$ with a $K/k$-involution $\sigma$. Let $x\mapsto \overline{x}$ denote
    the nontrivial $k$-automorphism of $K$. For an \emph{admissible pair} $(u,\mu)$, i.e., $u\in (B,\sigma)_+$ such that $N_B(u)=\mu\overline{\mu}$ for some $\mu\in K^{\times}$, define a cubic norm structure on the $k$-vector space $(B,\sigma)_+\oplus B$ as follows:
$$N((b,x)):=N_B(b)+T_K(\mu N_B(x))-T_B(bxu\sigma(x)),$$
$$(b,x)^{\#}:=(b^{\#}-xu\sigma(x), \overline{\mu}\sigma(x)^{\#}u^{-1}-bx),~1:=(1_B,0).$$ 
The Jordan algebra obtained from this cubic norm structure is denoted by $J(B,\sigma,u,\mu)$. Note that $(B,\sigma)_+$ is a subalgebra of $J(B,\sigma, u,\mu)$ through the first summand. Then $J(B,\sigma,u,\mu)$ is a division algebra if and only if $\mu$ is not a norm from $B$ (see \cite{PR7}, 5.2). This construction is called the \emph{second Tits process} arising from the parameters $(B,\sigma),u$ and $\mu$.

When $K=k\times k$, then $B=D\times D^{\circ}$ and $\sigma$ is the switch involution, where $D^{\circ}$ is the opposite algebra of $D$. In this case, the second construction $J(B,\sigma,u,\mu)$ can be identified with a first construction $J(D,\lambda)$. 
\vskip1mm
\noindent
{\bf Tits constructions of Albert algebras :} The Tits process starting with a central simple algebra $D$ and $\lambda\in k^{\times}$ yields the \emph{first Tits construction} Albert algebra $A=J(D,\lambda)$ over $k$. 
Similarly, in the Tits process if we start with a central simple algebra $(B,\sigma)$ with center a quadratic \'{e}tale algebra $K$ over $k$ and an involution $\sigma$ of the second kind, $u,\mu$ as described above, we get the \emph{second Tits construction} Albert algebra $A=J(B,\sigma,u,\mu)$ over $k$. 
\noindent
One knows that all Albert algebras can be obtained via Tits constructions. 

An Albert algebra is a division algebra if and only if its (cubic) norm $N$ is anisotropic over $k$ (see \cite{KMRT}, \S 39).
If $A=J(B,\sigma,u,\mu)$ as above, then $A\otimes_kK\cong J(B,\mu)$ as $K$-algebras, where $K$ is the center of $B$ (see \cite{KMRT}, 39.4).
\vskip1mm
\noindent 
Let $A$ be an Albert algebra over $k$. If $A$ arises from  the first construction, but does not arise from the second construction then we call $A$ a
    \emph{pure first construction} Albert algebra. Similarly, \emph{pure second construction} Albert algebras are defined as those which do not arise from the first Tits construction. 
\vskip1mm
\noindent
For an Albert division algebra $A$, any subalgebra is either $k$ or a cubic subfield of $A$ or of the form $(B,\sigma)_+$ for a degree $3$ central simple algebra $B$ with an involution $\sigma$ of the second kind over its center $K$, a quadratic \'{e}tale extension of $k$ (see \cite{J1}, Chap. IX, \S 12, Lemma 2,  \cite{PR5}). 
\vskip1mm
\noindent
    {\bf Norm similarities of Albert algebras :} Let $A$ be an Albert algebra over $k$ and $N$ its norm map. By a \emph{norm similarity} of $A$ we mean a bijective $k$-linear map $f:A\rightarrow A$ such that $N(f(x))=\nu(f)N(x)$ for all $x\in A$ and some $\nu(f)\in k^{\times}$. When $k$ is infinite, the notions of norm similarity and isotopy for degree $3$ Jordan algebras coincide (see \cite{J1}, Chap. VI, Thm. 6, Thm. 7).
    
\noindent
Let $A$ be an Albert algebra over $k$ with norm map $N$. For $a\in A$ the $U$-operator $U_a$ is given by $U_a(y):=T(a,y)a-a^{\#}\times y,~y\in A$. When $a\in A$ is invertible, one knows that $U_a$ is a norm similarity of $A$, in fact, for any $x\in A,~N(U_a(x))=N(a)^2N(x)$. 

For a central simple algebra $D$ of degree $3$ over a field $k$, the adjoint map $a\mapsto a^{\#}$ satisfies $N_D(a)=aa^{\#}=a^{\#}a,~\forall~a\in D$. Motreover, $(xy)^{\#}=y^{\#}x^{\#}$ for all $x,y\in D$. It also follows that $N(x^{\#})=N(x)^2$. 
\vskip1mm
\noindent
{\bf Algebraic groups arising from Albert algebras :} In this paper, for a $k$-algebra $X$ and a field extension $L$ of $k$, $X_L$ will denote the $L$-algebra $X\otimes_kL$. Let $A$ be an Albert algebra over $k$ with norm $N$ and $\overline{k}$ be an algebraic closure of $k$. It is well known that the full group of automorphisms 
$\text{\bf Aut}(A):=\text{Aut}(A_{\overline{k}})$ is a simple algebraic group of type $F_4$ defined over $k$ and all simple groups of type $F_4$ defined over $k$ arise this way .
We will denote the group of $k$-rational points of $\text{\bf Aut}(A)$ by $\text{Aut}(A)$. It is known that $A$ is a division algebra if and only if the norm form $N$ of $A$ is anisotropic (see \cite{Spr}, Thm. 17.6.5). Albert algebras whose norm form is isotropic over $k$, i.e. has a nontrivial zero over $k$, are called \emph{reduced}. 

The \emph{structure group} of $A$ is the full group $\text{\bf Str}(A)$ of norm similarities of $N$, is a connected reductive group over $k$, of type $E_6$. We denote by $\text{Str}(A)$ the group of $k$-rational points $\text{\bf Str}(A)(k)$.


The automorphism group $\text{\bf Aut}(A)$ is the stabilizer of $1\in A$ in $\text{\bf Str}(A)$.
In the paper, base change of an object $X$ defined over a base field $k$ to an extension $L$ of $k$ will be denoted by $X_L$. 
\vskip1mm
\noindent

\vskip1mm
\noindent
{\bf $R$-equivalence and $R$-triviality:}
Let $X$ be an irreducible variety over a field $k$ with $X(k)\neq \emptyset$. We define points $x,y\in X(k)$ to be $R$-equivalent 
if there exists a sequence $x_0=x, x_1,\cdots, x_n=y$ of points in $X(k)$ and rational maps $f_i:\mathbb{A}_k^1\rightarrow X,~1\leq i\leq n$, defined over $k$ and regular at $0$ and $1$, such that $f_i(0)=x_{i-1},~f_i(1)=x_i$ (see \cite{M}). 

Let $G$ be a connected algebraic group defined over $k$. The set of points in $G(k)$ that are $R$-equivalent to $1\in G(k)$, is a normal subgroup of $G(k)$, denoted by $RG(k)$. The set $G(k)/R$ of $R$-equivalence classes in $G(k)$ is in canonical bijection with $G(k)/RG(k)$ and thus has a natural group structure. We identify $G(k)/R$ with the group $G(k)/RG(k)$. This group is useful in studying rationality properties of $G$. 

Call $G$ \emph{$R$-trivial} if $G(L)/R=\{1\}$ for all field extensions $L$ of $k$. A variety $X$ defined over $k$ is defined to be $k$-\emph{rational} if $X$ is birationally isomorphic over $k$ to an affine space. One knows that $G$ is $k$-rational then $G$ is $R$-trivial (see \cite{Vos}, Chap. 6, Prop. 2).  Also, owing to the fact that tori of rank at most $2$ are rational, one knows that algebraic groups of rank at most $2$ are rational. 


\section{\bf Automorphisms of Albert algebras}
We continue to set up notation for the rest of the paper, again mostly following (\cite{Th-2}). Let $A$ be an Albert algebra over $k$ and $S\subset A$ a subalgebra. Denote by {\bf Aut}$(A/S)$ the (closed) subgroup of $\text{\bf Aut}(A)$ consisting of automorphisms of $A$ which fix $S$ pointwise and $\text{\bf Aut}(A,S)$ denotes the closed subgroup of automorphisms of $A$ leaving $S$ invariant.
The group of $k$-rational points of these groups will be denoted by ordinary fonts, for example $\text{Aut}(A)= \text{\bf Aut}(A)(k)$. One knows that when $A$ is a division algebra,
a proper subalgebra $S$ of $A$ is either $k$ or a cubic subfield or a $9$-dimensional degree $3$ Jordan subalgebra, which is of the form $D_+$ for a degree $3$ central division algebra $D$ over $k$ or of the form $(B,\sigma)_+$ for a central division algebra $B$ of degree $3$ over a quadratic field extension $K$ of $k$, with a unitary involution $\sigma$ over $K/k$ (see Thm. 1.1, \cite{PR6}). We recall below a description of some subgroups of $\text{\bf Aut}(A)$ which will be used in the paper, see (\cite{KMRT}, 39.B) and (\cite{FP}) for details.
\begin{proposition}[Prop. 3.1, \cite{Th-2}]\label{rational} Let $A$ be an Albert algebra over $k$.
\begin{enumerate} 
\item Suppose $S=D_+\subset A$ for $D$ a degree $3$ central simple $k$-algebra. Write $A=J(D,\mu)$ for some $\mu\in k^{\times}$. Then $\text{\bf Aut}(A/S)\cong\text{\bf SL}(1,D)$, the algebraic group of norm $1$ elements in $D$. 
\item Let $S=(B,\sigma)_+\subset A$ for $B$ a degree $3$ central simple algebra of degree $3$ over a quadratic extension $K$ of $k$ with an involution $\sigma$ of the second kind over $K/k$. Write $A=J(B,\sigma,u,\mu)$ for suitable parameters. Then 
$\text{\bf Aut}(A/S) \cong\text{\bf SU}(B,\sigma_u)$, where $\sigma_u:=Int(u)\circ\sigma$. 
\end{enumerate}
In particular, the subgroups described in (1) and (2) are simply connected, simple of type $A_2$, defined over $k$ and since these are rank-$2$ groups, they are are rational and hence are $R$-trivial. 
\end{proposition}
\vskip1mm
\noindent
We need the following fixed point theorem.
\noindent
\begin{theorem}[Thm. 4.1, \cite{Th-2}]\label{fixedpoint} Let $A$ be an Albert division algebra over $k$ and $\phi\in\text{Aut}(A)$ an automorphism of $A$. Then $\phi$ fixes a cubic subfield of $A$ pointwise.
\end{theorem}
\noindent
\vskip1mm
\noindent
{\bf Extending automorphisms :} Let $A$ be an Albert division algebra and $S\subset A$ a subalgebra. We can extend automorphisms of $S$ to automorphisms of $A$ in some cases (e.g. when $S$ is nine dimensional, see \cite{P-S-T1}, Thm. 3.1, \cite{P}, Thm. 5.2). For the purpose of proving $R$-triviality results that we are after, we need certain explicit extensions, we proceed to describe these (cf. \cite{KMRT}, 39.B). 

First, let $S$ be a $9$-dimensional division subalgebra of $A$. Then we may assume $S=(B,\sigma)_+$ for a degree $3$ central simple algebra with center $K$, 
a quadratic \'{e}tale extension of $k$ (see \cite{PR6}). Then $A\cong J(B,\sigma, u,\mu)$ for suitable parameters.

Any automorphism of $S$ is of the form $x\mapsto pxp^{-1}$ with $p\in\text{Sim}(B,\sigma)$, where $\text{Sim}(B,\sigma)=\{g\in B^{\times}|g\sigma(g)\in k^{\times}\}$ (see Thm. 5.12.1, \cite{J3}) and (\cite{J-4}, Thm. 4). Let $\sigma_u:=\text{Int}(u)\circ\sigma$, where $\text{Int}(u)$ is the conjugation by $u$. We have,
\begin{proposition}[Prop. 3.2, \cite{Th-2}]\label{aut-ext} Let $A=J(B,\sigma, u,\mu)$ be an Albert algebra over $k$ with $S:=(B,\sigma)_+$ a division algebra and $K=Z(B)$. Let $\phi\in\text{Aut}(S)$ be given by $\phi(x)=gxg^{-1}$ for 
$g\in\text{Sim}(B,\sigma)$ with $g\sigma(g)=\lambda\in k^*$ and $\nu:=N_B(g)\in K^*$. Let $q\in U(B,\sigma_u)$ be arbitrary with $N_B(q)=\bar{\nu}^{-1}\nu$. Then the map $\widetilde{\phi}:A\rightarrow A$, given by 
$$(a,b)\mapsto (gag^{-1},\lambda^{-1}\sigma(g)^{\#}bq),$$ 
is an automorphism of $A$ extending $\phi$.  
\end{proposition}
\begin{corollary}[Cor. 3.1, \cite{Th-2}]\label{aut-ext-D} Let $A=J(D,\mu)$ be an Albert algebra arising from the first Tits construction, where $D$ is a degree $3$ central division algebra. Let $\phi\in\text{Aut}(D_+)$ be given by $\phi(x)=gxg^{-1}$ for $g\in D^{\times}$ and $x\in D_+$. Then the map
  $$(x,y,z)\mapsto (gxg^{-1}, gyh^{-1}, hzg^{-1}),$$
  for any $h\in D^{\times}$ with $N_D(g)=N_D(h)$, is an automorphism of $A$ extending $\phi$.
  \end{corollary}
\noindent
We have,
\begin{proposition}[Prop. 3.3, \cite{Th-2}]\label{aut-A-Bsigma} Let $A=J(B,\sigma, u, \mu)$ be an Albert algebra with $S:=(B,\sigma)_+$ a division algebra and $K=Z(B)$, a quadratic \'{e}tale extension of $k$. Then any automorphism of $A$ stabilizing $S$ is of the form $(a,b)\mapsto (pap^{-1}, pbq)$ for $p\in U(B,\sigma)$ and $q\in U(B,\sigma_u)$ with 
$N_B(p)N_B(q)=1$. We have 
$$\text{Aut}(A, (B,\sigma)_+)\cong [U(B,\sigma)\times U(B,\sigma_u)]^{det}/K^{(1)},$$
where
$$[U(B,\sigma)\times U(B,\sigma_u)]^{det}:=\{(p,q)\in U(B,\sigma)\times U(B,\sigma_u)| N_B(p)=N_B(q)\},$$
and $K^{(1)}$ denotes the group of norm $1$ elements in $K$, embedded diagonally.
\end{proposition}
\section{\bf Norm similarities and $R$-triviality}
In the last section, we discussed a recipe to extend automorphisms of a $9$-dimensional subalgebra of an Albert division algebra to an automorphism of the Albert algebra. In this section, we analyse norm similarities of an Albert algebra $A$ in the same spirit, discuss rationality and $R$-triviality properties of some subgroups of {\bf Str}$(A)$. 

Since any norm similarity of $A$ fixing the identity element $1$ of $A$ is necessarily an automorphism of $A$, it follows that for any $k$-subalgebra $S\subset A$ the subgroup {\bf Str}$(A/S)$ of all norm similarities that fix $S$ pointwise, is equal to the subgroup {\bf Aut}$(A/S)$ of {\bf Aut}$(A)$, i.e., {\bf Str}$(A/S)=${\bf Aut}$(A/S)$.

The (normal) subgroup of $\text{Str}(A)$ generated by the $U$-operators of $A$ is denoted by $\text{Instr}(A)$. We will denote by $\text{\bf Str}(A,S)$ the full subgroup of $\text{\bf Str}(A)$ consisting of all norm similarities of $A$ that leave $S$ invariant and $\text{Str}(A,S)=\text{\bf Str}(A,S)(k)$.

\noindent
{\bf Extending norm similarities :} Let $A$ be an Albert (division) algebra and $S$ a $9$-dimensional subalgebra of $A$. 
Given an element $\psi\in\text{Str}(S)$, we wish to extend it to an element of $\text{Str}(A)$. This may be achieved by invoking a result of Garibaldi-Petersson (\cite{GP}, Prop. 7.2.4), however, we need certain explicit extension for the purpose of proving $R$-triviality of the algebraic groups described in the introduction.
\begin{theorem}[Thm. 4.4, \cite{Th-2}]\label{ext-norm-sim} Let $A$ be an Albert division algebra over $k$ and $S\subset A$ be a $9$-dimensional subalgebra. Then every element $\psi\in\text{Str}(S)$ admits an extension $\widetilde{\psi}\in\text{Str}(A)$. Let 
$S=(B,\sigma)_+$ and $A=J(B,\sigma,u,\mu)$ for suitable parameters. Let $\psi\in\text{Str}(S)$ be given by $b\mapsto\gamma gb\sigma(g)$ for $b\in (B,\sigma)_+,~g\in B^{\times}$ and $\gamma\in k^{\times}$. Then for any $q\in U(B,\sigma_u)$ with $N_B(q)=N_B(\sigma(g)^{-1}g)$, the map $\widetilde{\psi}$ given by 
$$\widetilde{\psi}((b,x))=\gamma(gb\sigma(g),\sigma(g)^{\#}xq),$$
is a norm similarity of $A$ extending $\psi$. 
\end{theorem}
\begin{theorem}[Thm. 4.5, \cite{Th-2}]\label{Str-ABsigma+}
Let $A=J(B,\sigma,u,\mu)$ be an Albert division algebra, written as a second Tits construction and $S=(B,\sigma)_+$. The group $\text{Str}(A,S)$ consists of the maps 
$(b,x)\mapsto\gamma(gb\sigma(b),\sigma(g)^{\#}xq)$ where $\gamma\in k^{\times}, g\in B^{\times}$ are arbitrary and $q\in U(B,\sigma_u)$ satisfies $N_B(q)=N_B(\sigma(g)^{-1}g)$. 
We have,
$$\text{Str}(A,(B,\sigma)_+)\cong \frac{k^{\times}\times H_0}{K^{\times}},$$
where $H_0=\{(g,q)\in B^{\times}\times U(B,\sigma_u)|N_B(q)N_B(\sigma(g)^{-1}g)=1\}$ and $K^{\times}\hookrightarrow k^{\times}\times H_0$ via $\alpha\mapsto (N_K(\alpha)^{-1},\alpha,\overline{\alpha}^{-1}\alpha)$.
\end{theorem}
\begin{corollary}[Cor. 4.3, \cite{Th-2}]\label{str-ext-D} Let $A=J(D,\mu)$ be an Albert algebra arising from the first Tits construction, where $D$ is a degree $3$ central division algebra over $k$. Then the group $\text{Str}(A,D_+)$ consists of the maps
  $$(x,y,z)\mapsto \gamma(axb, b^{\#}yc, c^{-1}za^{\#}),$$
  where $a, b, c \in D^{\times},~\gamma\in k^{\times}$ with $N_D(a)=N_D(b)N_D(c)$. 
  \end{corollary}
\section{\bf $R$-triviality results} In this section, we prove results on $R$-triviality of various groups and also prove our main theorem. We need a few results from (\cite{Th-2}) which we recall below. We also remark at the outset that, since Albert algebras arising from the first Tits construction are either split or are division algebras and the automorphism group of the split Albert algebra is the split group of type $F_4$, hence is $R$-trivial, we may focus on proving $R$-triviality for automorphism groups of Albert \emph{division} algebras arising from the first Tits construction. 
\begin{theorem}[Thm. 5.1, \cite{Th-2}]\label{Str-R-triv} Let $A$ be an Albert (division) algebra. Let $S\subset A$ be a $9$-dimensional subalgebra. Then, with the notations as above, $\text{\bf Str}(A,S)$ is $R$-trivial. 
\end{theorem}
\begin{corollary}[Cor. 5.1, \cite{Th-2}]\label{Aut-R-triv-S} Let $A$ be an Albert (division) algebra and $S$ a $9$-dimensional subalgebra of $A$. Then $\text{Aut}(A,S)\subset R(\text{\bf Str}(A)(k))$.
\end{corollary}
\noindent

  \begin{theorem}\label{main} Let $A$ be an Albert division algebra over a field $k$ or arbitrary characteristic. Then $\text{Str}(A)=R\text{\bf Str}(A)(k)$.
  \end{theorem}
  
  We now proceed to prove the main result
  \begin{theorem} Let $A$ be an Albert division algebra arising from the first Tits construction, over a field $k$ of arbitrary characteristic. Then $G=\text{\bf Aut}(A)$ is $R$-trivial.
  \end{theorem}
 
 We recall from (\cite{Th-1})
 \begin{theorem}\label{cyclic-fixed} Let $A$ be an Albert division algebra arising from the first Tits construction over a field $k$ and $L\subset A$ a cyclic cubic subfield. Let $\phi\in\text{Aut}(A/L)$. Then $\phi$ is a product of two automorphisms, each stabilizing a $9$-dimensional subalgebra.
 \end{theorem}
 Let $L/k$ be a cyclic cubic field extension and $K/k$ a quadratic separable field extension. Let $\rho$ be a generator for $Gal(L/k)$ also let $\rho$ denote the $K$-linear field automorphism of $LK$ extending $\rho\in Gal(L/k)$.
Let $x\mapsto \overline{x}$ denote the nontrivial field automorphism of $K/k$. We prove that $\text{\bf Aut}(A,S)\subset R\text{\bf Aut}(A)(k)$ for certain distinguished $9$-dimensional subalgebras $S\subset A$. This result works over fields of arbitrary characteristics and may be of independent interest. More precisely, we have
 \begin{theorem}\label{dist-R-triv} Let $A$ be an Albert division algebra arising from the first Tits construction, over a field $k$ of arbitrary characteristic. Let $S\subset A$ be a $9$-dimensional subalgebra with trivial mod-$2$ invariant. Then
   $\text{Aut}(A,S)\subset R\text{\bf Aut}(A,S)(k)$.
   \end{theorem}
 \begin{proof} We first assume $S=D_+=(D\times D^{op}, \epsilon)_+$ for a degree $3$ central division algebra $D$ over $k$. Then, by (\cite{KMRT}, Th. 39.14 (2)), there is an isomorphism $\gamma:J(D,\nu)\cong A$ for some $\nu\in k^{\times}$ such that $\gamma(D_+)=S$. We therefore have the induced isomorphism $\gamma:\text{\bf Aut}(J(D,\nu))\cong\text{\bf Aut}(A)$ of algebraic groups defined over $k$ and $\gamma(\text{\bf Aut}(J(D,\nu))=\text{\bf Aut}(A,S)$. Hence, without loss of generality, we may assume $A=J(D,\nu)$ and $S=D_+\subset J(D,\nu)$ as the first summand. Let $\phi\in\text{\bf Aut}(A,S)$.
   By (\cite{KMRT}, Cor. 39.12), there exist $a, b\in D^{\times}$ with $N_D(a)=N_D(b)$ and we have
   $$\phi((x,y,z))=(axa^{-1}, ayb^{-1}, bza^{-1}),~\forall x, y, z\in D.$$
   By the computation in the proof of (\cite{Th-1}, Th. 5.4), we have
   $$\phi=\mathcal{J}_{ab^{-1}}\mathcal{I}_a,$$
   where $\mathcal{J}_c\in\text{Aut}(A)$ for $c\in\text{SL}(1,D)$ is given by
   $$\mathcal{J}_c((x,y,z))=(x, yc, c^{-1}zc)$$
   and $\mathcal{I}_d\in \text{Aut}(A)$ for $d\in D^{\times}$ is given by
   $$\mathcal{I}_d((x, y, z))=(dxd^{-1}, dyd^{-1}, dzd^{-1}),~\forall x,y,z\in D.$$
   Hence 
   $$\mathcal{J}_p((x,y,z))=(x, yp, p^{-1}z),~p=ab^{-1}\in\text{SL}(1,D),$$
   $$\mathcal{I}_a((x,y,z))=(axa^{-1}, aya^{-1}, aza^{-1}),~\forall x, y, z\in D.$$
   The group $\text{\bf SL}(1,D)$, having (absolute) rank $2$, is $R$-trivial. Let $\theta:\mathbb{A}_k^1\rightarrow\text{\bf Aut}(A,D_+)$ be defined by
   $$\theta(t)((x,y,z))=(a_txa_t^{-1}, a_tya_t^{-1}, a_tza_t^{-1}),~x,y,z\in D,$$
   where $a_t=(1-t)a+t\in D$. Then $\theta:\mathbb{A}_k^1\rightarrow\text{\bf Aut}(A,D_+)$ is a rational map, defined over $k$, regular at $0,1$ and $\theta(0)=\mathcal{I}_a,~\theta(1)=1$. Hence both $\mathcal{J}_{ab^{-1}}$ and $\mathcal{I}_a$ are in $R\text{\bf Aut}(A,D_+)$. It follows that $\phi\in R\text{\bf Aut}(A,D_+)$. 
   
   Next, we write $S=(B,\sigma)_+$ for a suitable central simple algebra $(B,\sigma)$ with a distinguished unitary involution $\sigma$ over its centre $K$, a quadratic separable field extension of $k$. Then $A=J(B,\sigma, u, \mu)$ for suitable admissible pair $(u,\mu)\in (B,\sigma)_+^{\times}\times K^{\times}$, with $N_B(u)=1$. Let $\phi\in\text{Aut}(A,S)$.
   By (\cite{KMRT}, Th. 40.2 (2)), We have
   $$f_3(A)=f_3(J(B,\sigma, u, \mu))=f_3((B,\sigma_u))=0=f_3(B,\sigma).$$
   So that $\sigma_u$ is also distinguished and hence is isomorphic to $\sigma$. Therefore, by (\cite{HKRT}, Lemma 1), there exists $w\in B^{\times}$ and $\lambda\in k^{\times}$ such that $u=\lambda w\sigma(w)$. We have, by (\cite{KMRT}, Lemma 39.2 (1))
   $$A=J(B,\sigma, u, \mu)=J(B, \sigma, \lambda w\sigma(w), \mu)\cong J(B, \sigma, \lambda, N(w)^{-1}\mu).$$
   But then we must have $N(\lambda)=\lambda^3=(N(w)^{-1}\mu)\overline{(N(w)^{-1}\mu)}$ and hence
   $$\lambda=(N(w)^{-1}\mu\lambda^{-1})\overline{(N(w)^{-1}\mu\lambda^{-1})}.$$
   Hence, again by (\cite{KMRT}, Lemma 39.2 (1)), we have an isomorphism
   $$\gamma: A=J(B,\sigma, u, \mu)\cong J(B, \sigma, \lambda, N(w)^{-1}\mu)\cong J(B, \sigma,1, \nu)$$
   where $\nu=\mu^{-2}N(w)^2\lambda^3$. Moreover, by (\cite{KMRT}, Lemma 39.2 (1)), this isomorphism, restricted to $(B,\sigma)_+$ is identity. It follows that the map $\gamma\phi\gamma^{-1}: J(B,\sigma, 1, \nu)\rightarrow J(B, \sigma, 1, \nu)$ is an automorphism and $\phi':=\gamma\phi\gamma^{-1}$ leaves the subalgebra $(B,\sigma)_+$ invariant. Clearly $\phi\in R\text{\bf Aut}(A,S)$ if and only if $\phi'\in R\text{\bf Aut}(A', S')$, where $A'=J(B, \sigma, 1, \nu)$ and $S'=(B,\sigma)_+\subset A'$.

   By (Prop. \ref{aut-A-Bsigma}), $\phi'=\psi_{p,q}$ for $p,q\in \text{U}(B, \sigma)$ with $N_B(p)=N_B(q)$. We have therefore
   $$\phi^{'}((b,x))=(pbp^{-1}, pxq^{-1}),~\forall b\in (B,\sigma)_+,~x\in B.$$
   Since $N_B(p)=N_B(q)$, we have $N_B(pq^{-1})=1$ and hence $pq^{-1}\in\text{SU}(B,\sigma)$. Thus we have
   $$\psi_{p,q}=\psi_{p,p}\psi_{1,pq^{-1}}.$$
   By (\cite{CM}, Prop. 2.4), $\text{\bf U}(B,\sigma)$ is $R$-trivial and $\text{\bf SU}(B,\sigma)$ is $R$-trivial since it has absolute rank $2$. Hence $\phi'=\gamma\phi\gamma^{-1}\in R\text{\bf Aut}(A',S')(k)$. It follows that $\phi\in R\text{\bf Aut}(A,S)(k)$. This completes the proof.   
 \end{proof}
 For the case of an arbitrary $9$-dimensional subalgebra, we first settle a technical result.
 \begin{proposition}\label{sticky} Let $A$ be an Albert division algebra arising from the first Tits construction. Let $U\subset A$ be the Zariski open set
   $U=\{a\in A|k(a)~\text{is~\'{e}tale}\}$. Then for $a\in U$, there exists $v=v(a)\in L^{\perp}$, $v\neq 0$ such that $a\mapsto \lambda(a):=N(v(a))$ is a morphism $U\rightarrow \mathbb{A}_k^1$ defined over $k$.   
 \end{proposition}
 \begin{proof} That $U$ is Zariski open follows from (\cite{KMRT}, Lemma 38.2.2). The discriminant $\delta(a)$ of $a\in U$ over $k$ is a polynomial in the coefficient of the
   minimal polynomial of $a$ over $k$ and hence the map $a\mapsto\delta(a)$ is a morphism $U\rightarrow \mathbb{A}_k^1$. Assume first that $L=k(a)$ is a cubic subfield of $A$ and let $\delta:=\delta(a)$ be the discriminant of $a$ over $k$. By (\cite{PR2}), the Springer form $q_L$ of $L$ is similar to $<\delta>\perp n_{C_0}$ over $L$, where $C$ is the split octonion over $k$ and $n_C$ its norm form, $C_0$ the subspace of $C$ supporting the \emph{pure norm} $n_{C_0}$ of $C$ (see \cite{PR2}, Prop. 3.1, Th. 3.2).
   Note that, $n_{C_0}$ being isotropic, there exists $v\in C_0$ such that $v\neq 0$ and $n_{C_0}(v)=0$. Hence $q_L(v\otimes 1)=0$. 
   The case when $L=k(a)=k$, i.e. when $a\in k$, the quadratic form that maps $x\in L^{\perp}$ to the projection of $x^{\#}$ on $L$, coincides with the \emph{quadratic trace form} of $A$ (see e.g. \cite{PF}, (8)). We can thus
   think of the quadratic trace as the Springer form in this case. By (\cite{P}, 1.7.6), in the case of first Tits constructions and $L=k$, $n_C$ is a subform of $q_L$ over $L$, since the reduced model of $A$ is split in the case $A$ is a first Tits construction. We let $v(a)$ denote the vector $v(a):=v\otimes_kk(a)$. Then $v(a)\neq 0$
   and $v(a)$ depends only on $a$ and has coordinates in $k$. It follows that the map $a\mapsto v(a)$ is a morphism of varieties $:U\rightarrow A$ defined over $k$. The result now follows by composing this morphism with the norm map $N$ of $A$.    
   \end{proof} 
 \begin{theorem}\label{S-stable} Let $A$ be an Albert division algebra over a field $k$ of arbitrary characteristic, arising from the first Tits construction and $S\subset A$ a $9$-dimensional subalgebra. Then $\text{Aut}(A,S)\subset R\text{\bf Aut}(A)(k)$.
 \end{theorem}
 \begin{proof} Let $A$ and $S$ be as in the statement of the theorem and let $\phi\in \text{Aut}(A,S)$. Then, by (\cite{Th-2}, Th. 5.1), there exists a rational map $\theta:\mathbb{A}_k^1\rightarrow\text{\bf Str}(A,S)$, defined over $k$ and regular at $0,1$ with $\theta(0)=\phi$ and $\theta(1)=1$. Let $a_t:=\theta(t)(1)\in S$ and $L_t:=k(a_t)$. Let $v_t\in L_t^{\perp}$ be chosen as in Proposition \ref{sticky} and $\lambda_t:=N(v_t)$. Then the first Tits process $J(L_t, \lambda_t)$ is a subalgebra of $A$, which yields the cubic subfield $k(v_t)$ when $L_t=k$. Let $\chi_t\in\text{Str}(A, J(L_t, \lambda_t))$ be defined by
   $$\chi_t=R_{N_{L_t}(a_t)}U_{(0,0,1)}U_{(0, N_{L_t}(a_t)^{-1}, 0)},$$
   where $R_{N_{L_t}(a_t)}$ is the right homothety on $A$ by ${N_{L_t}(a_t)}$ and the other factors are $U$-operators on the first Tits process $J(L_t, \lambda_t)=L_t\oplus L_t\oplus L_t$, 
   acting on $A$ and leaving $J(L_t, \lambda_t)$ invariant. Then $\chi_0=1$. 
   Then, by (\cite{Th-2}, Lemma 5.1),
   $$\chi_t(\theta(t)(1))=\chi_t(a_t)=1.$$
   Therefore $\chi_t\theta(t)\in\text{Aut}(A)$ and the map $\eta: t\mapsto\chi_t\theta(t)$ is a rational morphism $\mathbb{A}_k^1\rightarrow\text{\bf Aut}(A)$ defined over $k$. Moreover, $\eta(0)=\chi_0\theta(0)=\phi$, since $\theta(0)=\phi$ and $\phi(1)=1$, so $a_0=\theta(0)(1)=\phi(1)=1$ and $\chi_1=1$. Also $\eta(1)=\chi_1\theta(1)=1$ since $\theta(1)=1$ and $a_1=\theta(1)(1)=1(1)=1$. This completes the proof.

 \end{proof}
 \begin{corollary}\label{aut-A-Lfix-Rtriv} Let $A$ be an Albert division algebra arising from the first Tits construction and $L\subset A$ be a cyclic cubic subfield. Then $\text{Aut}(A/L)\subset R\text{\bf Aut}(A)(k)$. 
   \end{corollary}
 \begin{proof} Let $\rho$ be a nontrivial generator for $Gal(L/k)$. Let $D_+\subset A$ be a subalgebra, $D$ a degree $3$ central division algebra, suchy that $L\subset D_+$. Then, by (\cite{Th-2}, Cor. 3.1), $\rho$ admits and extension $\widetilde{\rho}$ to an automorphism of $D_+$ and by (Cor. \ref{aut-ext-D}), to an automorphism of $A$, which we continue to denote by $\widetilde{\rho}$. Let now $\phi\in\text{Aut}(A/L)$. Then $\widetilde{\rho}^{-1}\phi\notin\text{Aut}(A/L)$. Let $M\neq L$ be a cubic subfield of $A$, fixed pointwise by $\widetilde{\rho}^{-1}\phi$, which is possible by Theorem \ref{fixedpoint}. Let $S:=<L,M>$ be the subalgebra of $A$ generated by $L$ and $M$. Then $\widetilde{\rho}^{-1}\phi$ leaves $S$ stable. By Theorem \ref{S-stable} $\widetilde{\rho}^{-1}\phi\in R\text{\bf Aut}(A)(k)$ and $\widetilde{\rho}(D_+)=D_+$, hence by the above proposition,
   $\widetilde{\rho}\in R\text{\bf Aut}(A)(k)$. Hence it follows that $\phi\in R\text{\bf Aut}(A)(k)$. 
 \end{proof}
 \begin{corollary}\label{aut-A-L-Rtriv} Let $A$ be an Albert division algebra arising from the first Tits construction and $L\subset A$ be a cyclic cubic subfield. Then $\text{Aut}(A,L)\subset R\text{\bf Aut}(A)(k)$. 
 \end{corollary}
 \begin{proof} Let $\phi\in\text{Aut}(A,L)$. We may assume, by the above corollary that $\phi|L\neq 1$. Let $D_+\subset A$ be the subalgebra corresponding to a degree $3$ central division algebra $D$ over $k$ such that $L\subset D_+$. Let $\rho$ denote the nontrivial Galois automorphism of $L/k$ and $\widetilde{\rho}$ denote its extension to an automorphism of $D_+$ and also denote the further extension of this automorphism of $D_+$ to an automorphism of $A$ as in (Cor. \ref{aut-ext-D}). Then $\widetilde{\rho}\in R\text{\bf Aut}(A)(k)$ and $\widetilde{\rho}^{-1}\phi\notin\text{Aut}(A/L)$. Let $M$ be the subfield of $A$ fixed pointwise by $\widetilde{\rho}^{-1}\phi$. Then $M\neq L$ and $S:=<L,M>$, the subalgebra of $A$ generated by $L$ and $M$ is $9$-dimensional and $\widetilde{\rho}^{-1}\phi\in\text{Aut}(A,S)$. By the above theorem, $\text{Aut}(A,S)\subset R\text{\bf Aut}(A)(k)$ and $\widetilde{\rho}^{-1}\in R\text{\bf Aut}(A)(k)$. Hence $\phi\in R\text{\bf Aut}(A)(k)$. 
   \end{proof}
 We need a few results from (\cite{Th-4}) for our purpose here, we reproduce some proofs for the sake of convenience.
 \begin{proposition}[\cite{Th-4}, Th. 6.4]\label{action} Let $A$ be an Albert division algebra over a field $k$ or arbitrary characteristic. Let $M\subset A$ be a cubic subfield and $\mathcal{S}_M$ denote the set of all $9$-dimensional subalgebras of $A$ that contain $M$. Then the group $\text{Str}(A,M)$ acts on $\mathcal{S}_M$.
 \end{proposition}
 \begin{proof} Let $\psi\in\text{Str}(A,M)$ and $S\in\mathcal{S}_M$ be arbitrary. Let $\phi\in\text{Aut}(A/S)$, $\phi\neq 1$. Then $A^{\phi}=S$. We have
   $$\psi\phi\psi^{-1}(1)=\psi(\psi^{-1}(1))=1,$$
   since $\psi^{-1}(1)\in M\subset S$ and $\phi$ fixes $S$ pointwise. Hence $\psi\phi\psi^{-1}\in\text{Aut}(A)$. It follows also that the fixed point subspace of $\psi\phi\psi^{-1}$ is precisely $\psi(S)$ and hence $\psi(S)$ is a subalgebra of $A$ of dimension $9$ and $M=\psi(M)\subset\psi(S)$, therefore $\psi(S)\in\mathcal{S}_M$. This completes the proof.
   \end{proof}
   \begin{proposition}[\cite{Th-4}, Claim 1, Proof of Th. 6.5]\label{S-isomorphic} Let the notation be as above, $\psi\in\text{Str}(A,M)$ and $S\in\mathcal{S}_M$. Then $S$ and $\psi(S)$ are $k$-isomorphic as subalgebras of $A$.   
   \end{proposition}
   \begin{proof} Let $f\in\text{Aut}(A/S),~f\neq 1$ be arbitrary. Then, for any $x\in S$ we have,
     $$\psi f\psi^{-1}(\psi(x))=\psi(f(x))=\psi(x).$$
     Also,
     $$\psi f\psi^{-1}(1)=\psi f(\psi^{-1}(1))=\psi(\psi^{-1}(1))=1, $$
     since $\psi^{-1}(1)\in M\subset S$ and $f$ fixes $S$ pointwise. Hence $\psi f\psi^{-1}\in\text{Aut}(A/\psi(S))$. The map $\text{\bf Aut}(A/S)\rightarrow\text{\bf Aut}(A/\psi(S))$ given by $f\mapsto \psi f\psi^{-1}$ is an isomorphism of algebraic groups that is defined over $k$. In fact, by the same argument, the conjugacy by $\phi$ gives an isomorphism $\text{\bf Aut}(A,S)\rightarrow \text{\bf Aut}(A,\psi(S))$. It follows that this maps the centralizer $Z_{\text{\bf Aut}(A,S)}(\text{\bf Aut}(A/S))$ isomorphically onto the centralizer $Z_{\text{\bf Aut}(A,\psi(S))}(\text{\bf Aut}(A/\psi(S))$.

     Both groups $\text{\bf Aut}(A/S)$ and $\text{\bf Aut}(A/\psi(S))$ are simple, simply connected outer forms of $A_2$. Writing $S=(C,\tau)_+$ for a suitable central simple algebra $C$ with a unitary involution $\tau$ and $\psi(S)=(C,\tau_w)_+$ for $w\in (C,\tau)_+$, by using the explicit realization of these groups as in (Prop. \ref{rational} and Prop. \ref{aut-A-Bsigma}), it follows by a computation that
     $$Z_{\text{\bf Aut}(A,S)}(\text{\bf Aut}(A/S))=\text{\bf SU}(C,\tau)~\text{and}~Z_{\text{\bf Aut}(A,\psi(S))}(\text{\bf Aut}(A/\psi(S))=\text{\bf SU}(C\tau_w).$$
     These groups are isomorphic by the aforementioned isomorphism, hence we have, by a Lie algebra argument (see \cite{J}, Chap. X, Th. 11), $S=(C,\tau)_+\cong (C,\tau_w)_+=\psi(S)$ as $k$-algebras.
    This completes the proof. 
   \end{proof}
   \begin{proposition}[\cite{Th-4}, proof of Claim 3, Th. 6.5]\label{key} Let $L\subset A$ be a cyclic cubic subfield. Let $\psi\in\text{Str}(A,L)$ and $S\in\mathcal{S}_L$, where $S=(B,\sigma)_+$ for a degree $3$ central division algebra $B$ over a quadratic \'{e}tale extension $K/k$ and $\sigma$ is a distinguished unitary involution on $B$. Then there exists $\phi\in\text{Aut}(A,L)$ such that $\phi^{-1}\psi\in \text{Str}(A,S)$.
   \end{proposition}
   \begin{proof} By (Prop. \ref{S-isomorphic}), we have $(B,\sigma)_+\cong \psi((B,\sigma)_+)$. Since $\sigma$ is distinguished, we have, by (\cite{P} and \cite{P2}), $\psi((B,\sigma)_+)=(B,\sigma_v)$ for some $v\in L$ with $N_L(v)=1$, since $\psi(L)=L$. Hence we have $(B,\sigma)_+\cong (B,\sigma_v)_+$. Since both $\sigma$ and $\sigma_v$ are distinguished involutions on $B$, by (\cite{KMRT}, Cor. 19.31),  there exists $w\in (LK)^{\times}$ and $\lambda\in k^{\times}$ such that $v=\lambda ww^*$, where $*$ is the nontrivial $k$-automorphism of $K/k$, extended $L$-linealy to an automorphism of $LK$.

     By (\cite{PT}, 1.6), the inclusion $L\subset (B,\sigma)_+$ extends to an isomorphism $\delta:J(LK, *, u, \mu)\rightarrow (B,\sigma)_+$ for a suitable admissible pair $(u,\mu)\in L^{\times}\times K^{\times}$. Moreover, $L\subset (B,\sigma_v)_+$, hence there is an isomorphism $\eta: J(LK, *, u', \mu')\rightarrow\psi(S)=(B,\sigma_v)_+$ for suitable $(u', \mu')$ extending the inclusion $L\hookrightarrow (B,\sigma_v)_+$. Since $\eta$ is identity on $L$, it follows that $J(LK, *, u', \mu')$ is the $v$-isotope of $J(LK, *, u, \mu)$. Hence, by (\cite{PR7}, Prop. 3.9), we may assume $(u', \mu')=(uv^{\#}, N_L(v)\mu)$. We have therefore
     $$u'=uv^{\#}=u(\lambda w w^*)^{\#}=u\lambda^2 w^{\#}(w^{\#})^*=uw_0w_0^*,~\mu'=N_L(v)\mu,$$
     where $w_0=\lambda w^{\#}$. We have an isomorphism (see \cite{PR7}, Prop. 3.7)
     $$\phi_1: J(LK, *, uv^{\#}, N_L(v)\mu)=J(LK, *, uw_0w_0^*, \mu)\rightarrow J(LK, *, u, N_{LK}(w_0)^{-1}\mu).$$
     We have
     $$N_{LK}(w_0)\overline{N_{LK}(w_0)}=N_{LK}(\lambda w^{\#})\overline{N_{LK}(\lambda w^{\#})}=(\lambda^3 N_{LK}(w)^2)\lambda^3\overline{N_{LK}(w)^2)}$$
     $$=(\lambda^3 N_{LK}(w)\overline{N_{LK}(w)})^2=N_L(v)^2=1.$$
     Hence, by (\cite{PT}, Lemma 4.5), there exists $w_1\in LK$ with $w_1 w_1^*=1$ and $N_{LK}(w_1)=N_{LK}(w_0)$.
     The map 
     $$\phi_2:J(LK, *, u, N_{LK}(w_0)^{-1}\mu)\rightarrow J(LK, *, u, \mu)$$
     given by $\phi_2((l,x))=(l, xw_1^{-1})$ is an isomorphism by (\cite{PR7}, Prop. 3.7). Hence the composite $\phi_1^{-1}\phi_2^{-1}: J(LK, *, u, \mu)\rightarrow J(LK, *, u', \mu')$ is an isomorphism which maps $L$ to $L$. We have therefore the isomorphism $\phi:=\eta(\phi_1^{-1}\phi_2^{-1})\delta^{-1}:S\rightarrow \psi(S)$, which satisfies $\phi(L)=L$ and $\phi^{-1}\psi(S)=S$, where $\phi$ denotes any extension of $\phi$ as above, to an automorphism of $A$, which is possible by (\cite{P-S-T1} and \cite{P}). This completes the proof. 
     \end{proof}
 \begin{theorem}\label{str-A-L-rtriv} Let $A$ be an Albert division algebra arising from the first Tits construction and $L\subset A$ a cyclic cubic subfield. Then 
   $\text{\bf Str}(A,L)\subset R\text{\bf Str}(A)$.
 \end{theorem}
 \begin{proof} Let $\psi\in\text{Str}(A,L)$. Let $D_+\subset A$ be a subalgebra corresponding to a degree $3$ central division algebra over $k$ such that $L\subset D_+$. Recall that $D_+=(D\times D^{op}, \epsilon)_+$, where $\epsilon$ is the switch involution and is distinguished (see \cite{KMRT}). Hence, by (Prop. \ref{key}), there exists an automorphism $\phi\in\text{Aut}(A,L)$ such that $\phi^{-1}\psi(D_+)=D_+$. Since, by (Cor. \ref{aut-A-L-Rtriv}), $\phi\in \text{Aut}(A,L)\subset R\text{\bf Aut}(A,L)\subset R\text{\bf Str}(A)(k)$, we may assume $\psi(D_+)=D_+$ and $\psi(L)=L$. We therefore have, by (Cor. \ref{str-ext-D}), 
   $$\psi((x,y,z))=\gamma(axb, b^{\#}yc, c^{-1}za^{\#})$$
   for some $a, b\in L^{\times},~\gamma\in k^{\times}$ and $c\in D^{\times}$ with $N_D(a)=N_D(b)N_D(c)$. Since scalar homotheties are in $R\text{\bf Str}(A)$, we may assume further that $\gamma=1$ and hence
   $$\psi((x,y,z))=(axb, b^{\#}yc, c^{-1}za^{\#}).$$
   Hence $N_D(ba^{-1}c)=1$ and therefore $c=(ab^{-1})d$ for $d\in\text{SL}(1,D)$. Conversely, for any $d\in\text{SL}(1,D)$ and $a, b\in L^{\times}$ arbitrary, the map $(x,y,z)\mapsto (axb, b^{\#}yc, c^{-1}za^{\#})$, where $c=ab^{-1}d$, is in $\text{Str}(A, D_+)$ and leaves $L$ stable. Hence we can identify $\psi$ as above, with the element
   $$((a,b), ab^{-1}d)\in L^{\times}\times L^{\times}\times (ab^{-1})\text{SL}(1,D).$$
   Note that the groups $R_{L/k}(\mathbb{G}_m)$ and $\text{\bf SL}(1,D)$ are both $R$-trivial. Hence $R_{L/k}(\mathbb{G}_m)\times R_{L/k}(\mathbb{G}_m)\times\text{\bf SL}(1,D)$ is $R$-trivial.
   More explicitly, we proceed as follows. Since $N_D(a)=N_D(b)N_D(c)$, we have $d:=a^{-1}bc\in \text{SL}(1,D)$. Fix $\gamma:\mathbb{A}^1_k\rightarrow\text{\bf SL}(1,D)$ with $\gamma(0)=d,~\gamma(1)=1$. Define $\theta:\mathbb{A}^1_k\rightarrow\text{\bf Str}(A,D,L)$ by
   $$\theta(t)((x, y, z))=(a_txb_t, b_t^{\#}yc_t, c_t^{-1}za_t^{\#}),$$
   where
   $$a_t=(1-t)a+t,~b_t=(1-t)b+t,~c_t=(a_t b_t^{-1}d_t),$$
   and $d_t=\gamma(t)$. Then
   $$\theta(0)((x,y,z))=(axb, b^{\#}yab^{-1}d, d^{-1}ba^{-1}z a^{\#})=(axb, b^{\#}yc, c^{-1}b^{-1}aba^{-1}za^{\#})=\psi((x,y,z)), $$
   and
   $$\theta(1)((x,y,z))=(x, yc_1, c_1^{-1}z)=(x,yd_1, d_1^{-1}z)=(x,y, z),$$
   as $\gamma(1)=1$. Hence $\theta(0)\psi,~\theta(1)=1$. We have thus proved the theorem. 
 \end{proof}
 \begin{theorem}\label{main} Let $A$ be an Albert division algebra arising from the first Tits construction. Then $\text{\bf Aut}(A)$ is $R$-trivial.
 \end{theorem}
 \begin{proof} Let $\phi\in\text{Aut}(A)$. Let $L\subset A$ be a cyclic cubic subfield. If $\phi(L)=L$, then by (Cor. \ref{aut-A-L-Rtriv}), we are done. So assume $M:=\phi(L)\neq L$.
   Let $S=<L, M>$ be the subalgebra of $A$ generated by $L$ and $M$. Then $S$ is a $9$-dimensional subalgebra of $A$. As in the proof of (\cite{Th-2}, Thm.5.3),
   we construct $\eta\in\text{Str}(A,S)$ such that $\eta(L)=M$. Hence $\phi^{-1}\eta\in\text{Str}(A,L)$. Let $\psi:=\phi^{-1}\eta$. Then, by the argument in the proof of (\cite{Th-4}, Thm. 6.5, Claim 3), there exists $\pi\in\text{Aut}(A,L)$ such that $\pi^{-1}\psi\in\text{Str}(A,S)$. But $\eta(S)=S$. Hence we have
   $$S=\pi^{-1}\psi(S)=\pi^{-1}\phi^{-1}\eta(S)=\pi^{-1}\phi^{-1}(S).$$
   Therefore we have $\pi^{-1}\phi^{-1}\in\text{Aut}(A,S)\subset R\text{\bf Aut}(A)(k)$. Also $\text{Aut}(A,L)\subset R\text{\bf Aut}(A)(k)$. Hence it follows that $\phi\in R\text{\bf Aut}(A)(k)$. 
   \end{proof}

\vskip5mm
Indian Statistical Institute, Stat-Math.-Unit, 8th Mile Mysore Road, Bangalore-560059, India.
\center email: maneesh.thakur@gmail.com

\end{document}